\documentclass[12pt,reqno]{amsart}
\usepackage{amssymb}
\usepackage{tikz-cd}

\usepackage{amscd}
\usepackage{amsmath, amsfonts,amssymb,amsthm,mathrsfs,xypic}
\usepackage{graphicx}
\usepackage{fancybox}
\calclayout
\usepackage{hyperref}
\usepackage[all]{xy}
\usepackage{extarrows}

\usepackage{mathtools}
\usepackage{mathrsfs,stmaryrd,wasysym}
\usepackage{graphics}
\usepackage{xypic}
\usepackage{multirow}
\usepackage{makecell}

\usepackage[english]{babel}
\usepackage{amsmath}

\allowdisplaybreaks[4]

\usepackage{cases}

\usepackage{dsfont}
\usepackage{color}
\SelectTips{cm}{}
\calclayout

\numberwithin{equation}{section}

\theoremstyle{plain}
\newtheorem{theorem}[equation]{Theorem}
\newtheorem{proposition}[equation]{Proposition}
\newtheorem{lemma}[equation]{Lemma}
\newtheorem{corollary}[equation]{Corollary}
\newcommand{\Hom}{\operatorname{Hom}}
\theoremstyle{definition}
\newtheorem{definition}[equation]{Definition}
\newtheorem{example}[equation]{Example}

\theoremstyle{remark}
\newtheorem{remark}[equation]{Remark}

\newcommand{\xrto}{\xrightarrow}

\newcommand{\cok}{\mathrm{cok}}

\begin{document}
	
	\title[]{$\tau$-tilting modules over trivial extensions}
\thanks{$*$ is the corresponding author. Both of the authors are supported by NSFC (Nos.11671174, 12171207). Zhang is also supported by the Project Funded by the Priority Academic Program Development of Jiangsu Higher Education Institutions and the Starting Project of Jiangsu Normal University.}
\begin{abstract} We study (support) $\tau$-tilting modules over the trivial extensions of finite dimensional algebras. More precisely, we construct two classes of (support)$\tau$-tilting modules in terms of the adjoint functors which extend and generalize the results on (support) $\tau$-tilting modules over triangular matrix rings given by Gao-Huang.
\end{abstract}

\author{Zhi-Wei Li}
\address{Z-W. Li: School of Mathematics and Statistics, Jiangsu Normal University,
Xuzhou, 221116, P. R. China}
\email{zhiweili@jsnu.edu.cn}
\author{Xiaojin Zhang$^*$}
\address{X. Zhang: School of Mathematics and Statistics, Jiangsu Normal University,
Xuzhou, 221116, P. R. China}
\email{xjzhang@jsnu.edu.cn; xjzhangmaths@163.com}
	\maketitle

	\setcounter{tocdepth}{1}
\section{Introduction}

In 2014, Adachi, Iyama and Reiten \cite{AIR} introduced the $\tau$-tilting theory as a generalization of the classical tilting theory in terms of mutations. It is shown that $\tau$-tilting theory are closed related to cluster tilting theory and silting theory. In $\tau$-tilting theory, support $\tau$-tilting modules are very essential. Therefore it is interesting to classify support $\tau$-tilting modules for given algebras. Many scholars focus on this topics. Mizuno \cite{Mi} classified support $\tau$-tilting modules over preprojective algebras; Adachi \cite{A1} classified support $\tau$-tilting modules over Nakayama algebras; Adachi \cite{A2} and Zhang \cite{Z} studied $\tau$-rigid modules over algebras of radical square zero; F. Eisele, G. Janssens and T. Raedschelders \cite{EJR} studied the $\tau$-rigid modules over special biserial algebras; Iyama and Zhang \cite{IZ} classified the support $\tau$-tilting modules over the Auslander algebra of $K[x]/(x^n)$; Zito \cite{Zi} studied the support $\tau$-tilting modules over cluster-tilted algebras. In particular, Gao and Huang \cite{GH} studied support $\tau$-tilting modules over lower triangular matrix rings; Peng, Ma and Huang \cite{PMH} also considered support $\tau$-tilting modules over lower triangular matrix rings. Moreover, Zhang \cite{Zh} studied the support $\tau$-tilting modules over lower triangular matrix rings and generalized the results of Gao and Huang in terms of recollement of abelian categories.

On the other hand, the trivial extension of algebras can be back to Fossum-Griffith-Reiten's trivial extension of abelian categories \cite{FGR}. Later it has gained attention of algebraists. Tachikawa \cite{T} classified the indecomposable representation of the trivial extension by $\mathbb{D}\Lambda$ over hereditary algebras of finite representation type. Later Yamagata \cite{Y} studied the trivial extension by $\mathbb{D}\Lambda$ over tilted algebras of Dynkin type. Moreover, Happel \cite{H} built a connection between the bounded derived category and the trivial extensions of algebras. Recently Assem, Brustle, Schiffler and Todrov \cite{ABST} showed that some special trivial extension of algebras are closed related to cluster categories. For applications of trivial extension of abelian categories in ring theory, we refer to \cite{AnF,M}.

In the present paper, we study the support $\tau$-tilting modules over trivial extensions of finite dimensional algebras in terms of adjoint pairs. We have the following main result (See subsection 2.3 for the definitions of the functors $T$ and $Z$).

\begin{theorem}\label{1} $(\mathrm{Theorem}\ \ref{main theorem}\ )$ Let $A$ be a finite dimensional algebra over a field $K$ and let $\Lambda$ be the trivial extension of $A$ by a  finitely generated $A$-bimodule $M$. For a pair $(X, P)\in A\mbox{-}\mathrm{mod}$, we have
	
	$(1)$ $(T(X),T(P))$ is a support $\tau$-tilting pair in $\Lambda\mbox{-}\mathrm{mod}$ if and only if $(X, P)$ is a support $\tau$-tilting pair in $A\mbox{-}\mathrm{mod}$, $\Hom_A(P,M\otimes_A X)=0$ and $\Hom_A(M\otimes_AX, \tau X)=0$.	
	
	$(2)$ $(Z(X), T(P))$ is a support $\tau$-tilting pair in $\Lambda\mbox{-}\mathrm{mod}$ if and only if $(X, P)$ is a support $\tau$-tilting pair $A\mbox{-}\mathrm{mod}$ and $\Hom_A(Q,X)=0$, where $Q$ is a projective cover of $M\otimes_AX$.
\end{theorem}

As a result of Theorem \ref{1}, we can generalize and extend the results on support $\tau$-tilting modules over lower triangular matrix rings by Gao-Huang \cite[Theorem 4.3]{GH}. Moreover, we have the following results on $\tau$-rigid modules over trivial extensions and lower triangular matrix algebras.

\begin{theorem}\label{2}$( \mathrm {Proposition}\  \ref{prop:necessary condition}, \mathrm{Corollary}\  \ref{cor:necessity condition for lower triangular ring}\ )$ $(1)$ Let $\Lambda$ be the trivial extension of a finite dimensional algebra $A$ by a finitely generated bimodule $M$. If $M\otimes_RX\xrto{\alpha}X$ is a $\tau$-rigid module in $\Lambda\mbox{-}\mathrm{mod}$, then both $X$ and $\cok \alpha$ are $\tau$-rigid modules in $A\mbox{-}\mathrm{mod}$.

$(2)$ Let $R$ and $S$ be finite dimensional algebras and $_SM_R$ a finitely generated bimodule. Let $\Lambda$ be the lower triangular matrix algebra
	$\left(\begin{smallmatrix}
		R&0\\
		M&S
	\end{smallmatrix}\right)$. If $(0, M\otimes_RX)\xrto{(0,\alpha)}(X,Y)$ is a $\tau$-rigid module in $\Lambda\mbox{-}\mathrm{mod}$, then $X$ is a $\tau$-rigid module in $R\mbox{-}\mathrm{mod}$ and $\cok \alpha$ is a $\tau$-rigid module in $S\mbox{-}\mathrm{mod}$.
\end{theorem}

Now we state the organization of the paper as follows:

In Section 2, we recall the basic results on trivial extensions of rings. In Section 3, we study the $\tau$-rigid modules and support $\tau$-tilting modules over trivial extensions of algebras and prove Theorems \ref{1} and \ref{2}. In Section 4, we give examples to illustrate our main results.

Throughout this paper, $\tau$ is the Auslander-Reiten translation functor. For an algebra $\Lambda$, we use $\Lambda\mbox{-}\mathrm{mod}$ to denote the category of finitely generated left $\Lambda$-modules.

{\bf Acknowledgement} The authors thank Xiao-Wu Chen and Zhibing Zhao for useful discussion on computing the representations of trivial extension of algebras. The second author would like to thank Yingying Zhang for useful comments on support $\tau$-tilting modules over lower triangular matrix algebras. The authors also want to thank the anonymous referee for showing them a key lemma and other useful suggestions to improve the paper.
	
\section{The categories of modules over trivial extensions of rings}

In this section, we recall some basic results on modules over trivial extensions of rings.

\subsection{Trivial extensions of rings} Let $R$ be a ring with $1$ and $M$ an $R$-bimodule. Denote by $R\ltimes M$, called the {\it trivial extension} of $R$ by $M$, to be the ring whose additive group is the direct sum $R\times M$ with multiplication given by $(r,m)(r',m')=(rr', rm'+mr')$.

It is well-known that a left $R\ltimes M$-module is a morphism $M\otimes_R X\xrto{\alpha}X$ such that the composition $M\otimes_RM\otimes_RX\xrto{M\otimes_R\alpha}M\otimes_RX\xrto{\alpha}X$ is zero.  If $M\otimes_RX\xrto{\alpha} X$ and $M\otimes_RY\xrto{\beta} Y$ are two left $R\ltimes M$-modules, then a morphism $f\colon \alpha\to \beta$ is a morphism $f\colon X\to Y$ such that the diagram

$$\CD
  M\otimes_RX @>{\alpha}>> X \\
  @V {I_M\otimes f} VV @V f VV  \\
  M\otimes_RY @>\beta>> Y
\endCD$$
is commutative. The composition of morphisms coincides with the composition in the category of left $R$-modules.

In the sequel, sometimes we use a pair $(X,\alpha)$ to denote the left $R\ltimes M$-module $M\otimes_RX\xrto{\alpha}X$.
\subsection{The lower triangular matrix rings} Let $R$ and $S$ be rings and let $_{S}M_R$ be a bimodule. Then the lower triangular matrix ring $\left(\begin{smallmatrix}
	R&0\\
	M&S
\end{smallmatrix}\right)$ is the ring with addition $\left(\begin{smallmatrix}
r&0\\
m&s
\end{smallmatrix}\right)+\left(\begin{smallmatrix}
r'&0\\
m'&s'
\end{smallmatrix}\right)=\left(\begin{smallmatrix}
r+r'&0\\
m+m'&s+s'
\end{smallmatrix}\right)$ and product $\left(\begin{smallmatrix}
r&0\\
m&s
\end{smallmatrix}\right)\left(\begin{smallmatrix}
r'&0\\
m'&s'
\end{smallmatrix}\right)=\left(\begin{smallmatrix}
rr'&0\\
mr'+sm'&ss'
\end{smallmatrix}\right).$

	Recall that the lower triangular matrix ring $\left(\begin{smallmatrix}
	R&0\\
	M&S
\end{smallmatrix}\right)$ is isomorphic to the trivial extension $(R\times S)\ltimes M$, where the right $R\times S$-module structure of $M$ is given via $R\times S\to R$ and the left $R\times S$-module structure of $M$ is given via $R\times S\to S$. A left $R\times S$-module is an order pair $(X,Y)$ with $X\in R\mbox{-}\mathrm{Mod}$ and $Y\in S\mbox{-}\mathrm{Mod}$. In fact, if $L$ is a left $R\times S$-module, then $X=(1,0)L$ is a left $R$-module and $Y=(0,1)L$ is a left $S$-module. Conversely, given a pair $(X,Y)$ in $(R\mbox{-}\mbox{Mod}, S\mbox{-}\mathrm{Mod})$, define $(r,s)(x,y)=(rx,sy)$, then $(X,Y)$ obtains a structure of left $R\times S$-module. Therefore, a left $\left(\begin{smallmatrix}
R&0\\
M&S
\end{smallmatrix}\right)$-module is of the form $M\otimes_{R\times S}(X,Y)=(0,M\otimes_RX)\xrto{(0,\alpha)} (X,Y).$
Hence, a left $\left(\begin{smallmatrix}
	R&0\\
	M&S
\end{smallmatrix}\right)$-module is determined uniquely by a morphism $M\otimes_RX\xrto{\alpha} Y$ of left $S$-modules.

\subsection{Adjoint pairs} For simplicity, we use $F$ to denote the tensor functor $M\otimes_R-$. There are pairs of adjoint pairs $(T,U)$ and $(C,Z)$, see \cite[Proposition 1.3]{FGR}.
\[\begin{tikzpicture}
	\node at(-3,0){$R\mbox{-}\mathrm{Mod}$};
	\node at(0,0){$R\ltimes M\mbox{-}\mathrm{Mod}$};
	\node at(3,0){$R\mbox{-}\mathrm{Mod}$};
	\node at(-1.6,.5){$T$};
	\node at(-1.6,-.5){$U$};
	\node at (1.7,.5){$C$};
	\node at(1.7,-.5){$Z$};
	\draw[->](-2.2,.2)--(-1.2,.2);
	\draw[->](-1.2,-.2)--(-2.2,-.2);
	\draw[->](1.2,.2)--(2.2,.2);
	\draw[->](2.2,-.2)--(1.2,-.2);
\end{tikzpicture}\]
The functor $T$ is defined on objects by
$$T(N)=\left(\begin{smallmatrix}
	0&0\\
	1&0
\end{smallmatrix}\right)\colon F(N)\oplus F^2(N)\to N\oplus F(N)$$
and on morphisms by
$$T(f)=\left(\begin{smallmatrix}
	f&0\\
	0&F(f)
\end{smallmatrix}\right), f:N\rightarrow L.$$
Thus, given a left $R\ltimes M$-module $(X,\alpha)$, it is easy to get that
$$\Hom_{R\ltimes M}(T(P),(X,\alpha))=\{(f,\alpha F(f)) \ | \ f\in\Hom_R(P,X)\}  \eqno(*)$$
for any left $R$-module $P$.

The functor $U$ is defined by $U(F(N)\xrto{\alpha}N)=N$ and $U(f)=f$. The functor $Z$ is given by $Z(N)=F(N)\xrto{0}N$ and $Z(f)=f$. The functor $C$ is defined by
$C(F(N)\xrto{\alpha}N)=\cok\alpha$ and $C(f)$ is the induced morphism.

Recall that an epimorphism $f\colon N\to L$ in an abelian category is called {\it minimal} if for any morphism $g\colon N'\to N$, $fg$ is surjective implies $g$ is an epimorphism.

\begin{lemma} \label{lem: adjoint} \cite[Proposition 1.5]{FGR} The functors $T$ and $C$ are right exact and $U$ and $Z$ are exact. Moreover
	\begin{itemize}
		\item $CT=\mathrm{Id}$;
		
		\item $UZ=\mathrm{Id}$;
		
		\item 	the unit $\eta\colon \mathrm{Id}\to ZC$ is a minimal epimorphism.		
	\end{itemize}
	\end{lemma}

\begin{remark}
	Let $R$ be a finite dimensional algebra and $M$ a finitely generated $R$-bimodule. For any $N\in R\mbox{-}\mathrm{mod}$, the morphism $(1,0)\colon T(N)\to Z(N)$ is a minimal epimorphism. So the projective covers of $T(N)$ and $Z(N)$ coincide.
\end{remark}

\subsection{The projective modules of trivial extensions of rings} By \cite[Corollary 1.6]{FGR}, a left $R\ltimes M$-module $M\otimes_RX\xrto{p} X$ is projective if and only if $C(p)$ is a projective $R$-module and $p\cong T(C(p))$.

Consider the exact sequence
$$F(X)\xrto{\alpha}X\xrto{\pi}C(\alpha)\to 0.$$ Since $P$ is projective, there is a morphism $q\colon P\to X$ such that $\pi q=p$. Let $\omega$ denote the morphism $(q,\alpha F(q))\colon P\oplus F(P)\to X$. Since $\alpha F(\alpha)=0$, we have that $(q,\alpha F(q))\left(\begin{smallmatrix}
	0&0\\
	1&0
\end{smallmatrix}\right)=\alpha(F(q),F(\alpha)F^2(q))$. Hence $\omega$ is a morphism from $T(P)$ to $(X,\alpha)$ in $R\ltimes M\mbox{-}\mathrm{Mod}$.
Now we recall the following lemmas on the projective presentation of a finitely generated module over trivial extensions.

\begin{lemma}\cite[Corollary 1.7(1)]{FGR}\label{projective cover}  Let $R$ be a finite dimensional algebra and $M$ a finitely generated $R$-bimodule. Let $(X,\alpha)$ be a module in $R\ltimes M\mbox{-}\mathrm{mod}$. If $P$ is the projective cover of $C(\alpha)$, then $T(P)$ is the projective cover of $(X,\alpha)$.
\end{lemma}

\begin{lemma} \label{lem:minimal proj presentation}  \cite[Proposition 2.8]{M} Let $R$ be a finite dimensional algebra with $X\in R\mbox{-}\mathrm{mod}$ and $M$ a finitely generated $R$-bimodule. If $Q\xrto{\xi}P\xrto{\zeta}X\to 0$ is a minimal projective presentation of $X$, then $T(Q)\xrto{T(\xi)}T(P)\xrto{T(\zeta)}T(X)\to 0$ is a minimal projective presentation of $T(X)$ in $R\ltimes M\mbox{-}\mathrm{mod}$.
\end{lemma}

\section{Support $\tau$-tilting modules}
In this section, we study the support $\tau$-tilting modules over trivial extensions of algebras in terms of adjoint functors. As a result, we can give two construction methods for support $\tau$-tilting modules over lower triangular matrix algebras.

In what follows, we assume that $A$ is a finite dimensional algebra and all modules are finitely generated. Denote by $\Lambda=A\ltimes M$ the trivial extension of $A$ by a bimodule $M$. Let the functors $F, T, C, U, Z$ be as in Section 2. For a left $A$-module $N$, denote by $|N|$ the number of non-isomorphic indecomposable direct summands of $N$.

We first recall the definition of support $\tau$-tilting modules in \cite{AIR}.

\begin{definition}\label{2.1} Let $N$ be a left $A$-module.

(1) $N$ is called {\it $\tau$-rigid} if $\mathrm{Hom}_{A}(N,\tau N)=0$.

(2) $N$ is called {\it $\tau$-tilting} if $N$ is $\tau$-rigid and $|N|=|A|$.

(3) $N$ is called {\it support $\tau$-tilting} if $N$ is $\tau$-tilting over the algebra $A/(e)$, where $e$ is an idempotent of $A$.

\end{definition}

We also need the following characterization of $\tau$-rigid modules.
\begin{lemma}\label{lem:tau-rigid module}\cite[Proposition 2.4 (3)]{AIR} Let $X$ be a left $A$-module and let $P_1\xrto{f}P_0\to X\to 0$ be a minimal projective presentation. Then $X$ is a $\tau$-rigid module if and only if $\Hom_A(f,X)$ is surjective. 	
\end{lemma}

We begin with a necessary condition on $\tau$-rigid $\Lambda$-modules.
\begin{proposition} \label{prop:necessary condition} Let $(X,\alpha)$ be a left $\Lambda$-module. If $(X,\alpha)$ is a $\tau$-rigid module, then $\cok \alpha$ is a $\tau$-rigid $A$-module.
\end{proposition}
\begin{proof}
	For $(X,\alpha)$, by \cite[Corollary 1.6]{FGR} we may assume its minimal projective presentation to be of the form
	$$T(P_1)\xrto{\left(\begin{smallmatrix}
			x & 0\\
			y & F(x)
		\end{smallmatrix}\right)}T(P_0)\xrto{\zeta=(q,\alpha F(q))}(X,\alpha)\to 0.$$
	By Lemma \ref{lem:tau-rigid module}, $(X,\alpha)$ is a $\tau$-rigid module if and only if $\Hom_{\Lambda}(\left(\begin{smallmatrix}
		x & 0\\
		y & F(x)
	\end{smallmatrix}\right), (X,\alpha))$ is surjective if and only if for any $f\in \Hom_A(P_1,X)$ there is $g\in \Hom_A(P_0,X)$ such that $f=gx+\alpha F(g)y$ by the equality $(*)$ in Section 2.
	
	Since the functor $C$ is right exact, it follows that $P_1\xrto{x}P_0\xrto{C(\zeta)} \cok \alpha\to 0$ is a projective presentation of $\cok \alpha$. Moreover $p:=C(\zeta)$ is right minimal by Lemma \ref{projective cover}. We write $x=(r,0)\colon P_1=P_1'\oplus P_1''\to P_0$, where $r$ is right minimal. Then the sequence
	$$P_1'\xrto{r}P_0\xrto{p}\cok \alpha\to 0$$
	is a minimal projective presentation of $\cok \alpha$.
	
	By Lemma \ref{lem:tau-rigid module}, $\cok\alpha$ is a $\tau$-rigid $A$-module if and only if $\Hom_A(r, \cok \alpha)$ is surjective. If $f$ is a morphism from $P_1'$ to $\cok \alpha$, then there exists a morphism $f'\colon P_1'\to X$ such that $f=\pi f'$ because $P_1'$ is projective and $\pi\colon X\to \cok\alpha$ is surjective. We extend $f'$ to a morphism $(f',0)\colon P_1=P_1'\oplus P_1'\to X$. Since $(X,\alpha)$ is a $\tau$-rigid $\Lambda$-module, by the previous discussion, there is a morphism $g\colon P_0\to X$ such that $(f',0)=gx+\alpha F(g)y.$
So we have $(f,0)=\pi(f',0)=\pi gx+\pi\alpha F(g)y=(\pi gr,0)$. In particular, $f=\pi gr$, which completes the proof.
\end{proof}

Let $R$ and $S$ be finite dimensional algebras and let $_SM_R$ be a finitely generated $S$-$R$-bimodule. Let $\Lambda$ be the lower triangular matrix algebra
$\left(\begin{smallmatrix}
	R&0\\
	M&S
\end{smallmatrix}\right)$. Then $\Lambda$ is isomorphic to the trivial extension $(R\times S)\ltimes M$ of $R\times S$ by $M$.
By Proposition \ref{prop:necessary condition}, we can give a necessary condition for $\tau$-rigid $\Lambda$-modules over lower triangular matrix algebras.
\begin{corollary} \label{cor:necessity condition for lower triangular ring} Let $\Lambda$ be the lower triangular matrix algebra
	$\left(\begin{smallmatrix}
		R&0\\
		M&S
	\end{smallmatrix}\right)$.\\ If ${(0, M\otimes_RX)\xrto{(0, \alpha)}(X, Y)}$ is a $\tau$-rigid $\Lambda$-module, then $X$ is a $\tau$-rigid $R$-module and $\cok \alpha$ is a $\tau$-rigid $S$-module.
	\end{corollary}

We can now formulate the following results on $\tau$-rigid modules.

\begin{proposition}\label{prop:tau-rigid module of trivial exten} Let $X$ be a left $A$-module. Then
	
	$(1)$ $T(X)$ is a $\tau$-rigid $\Lambda$-module if and only if $X$ is a $\tau$-rigid $A$-module and $\Hom_A(M\otimes_AX, \tau X)=0$.	
	
	$(2)$ $Z(X)$ is a $\tau$-rigid $\Lambda$-module if and only if $X$ is a $\tau$-rigid $A$-module and $\Hom_A(Q,X)=0$ where $Q$ is a projective cover of $M\otimes_AX$.
\end{proposition}
\begin{proof} (1)	Let $P_1\xrto{f}P_0\to X\to 0$ be a minimal projective presentation of $X$. Then $T(Q)\xrto{T(f)}T(P)\to T(X)\to 0$ is a minimal projective presentation of $T(X)$ in $\Lambda\mbox{-}\mathrm{mod}$ by Lemma \ref{lem:minimal proj presentation}. Therefore $T(X)$ is a $\tau$-rigid module if and only if $\Hom_{\Lambda}(T(f),T(X))$ is surjective. Notice that $(T,U)$ is an adjoint pair, this is equivalent to that $\Hom_A(f,X\oplus M\otimes X)$ is surjective. Since  $\Hom_A(f,X\oplus M\otimes_AX)=\Hom_A(f,X)\oplus \Hom_A(f,M\otimes X)$, it follows that $T(X)$ is a $\tau$-rigid $\Lambda$-module if and only if $X$ is a $\tau$-rigid $A$-module and $\Hom_A(M\otimes_AX, \tau X)=0$ by Lemma \ref{lem:tau-rigid module} and \cite[Proposition 2.4 (2)]{AIR}.

(2) Let $P_1\xrto{f}P_0\xrto{p} X\to 0$ be a minimal projective presentation of $X$. Suppose $q\colon Q\to F(X)$ is a projective cover. Since $F$ is right exact, the morphism $F(p)$ is an epimorphism. Therefore there is a morphism $r\colon Q\to F(P_0)$ such that $F(p) r=q$.

Since $CZ(X)=X$ and $p\colon P_0\to X$ is a projective cover, we know that $T(P_0)\xrto{(p,0)}Z(X)$ is a projective cover. Let $u\colon \ker p\to P_0$ be the kernel of $p$ and $v\colon P_1\to \ker p$ the projective cover such that $f=uv$. Notice that $$\left(\begin{smallmatrix}
	u&0\\
	0&1
\end{smallmatrix}\right)\colon (\ker p\oplus F(P_0), \left(\begin{smallmatrix}
		0&0\\
		F(u)&0
	\end{smallmatrix}\right))\to T(P_0)$$ is the kernel of $(p,0)$. Since $\left(\begin{smallmatrix}
v&0\\
0&q
\end{smallmatrix}\right)\colon P_1\oplus Q\to \ker p\oplus F(X)=C\left(\begin{smallmatrix}
0&0\\
F(u)&0
\end{smallmatrix}\right)$ is a projective cover, we know that
$$\left(\begin{smallmatrix}
	v&0&0&0\\
	0&r&F(f)&0
\end{smallmatrix}\right)\colon T(P_1\oplus Q)\to ( \ker p\oplus F(P_0), \left(\begin{smallmatrix}
0&0\\
F(u)&0
\end{smallmatrix}\right))$$
is a projective cover. Thus
$$T(P_1\oplus Q)\xrto{\left(\begin{smallmatrix}
		f&0&0&0\\
		0&r&F(f)&0
	\end{smallmatrix}\right)}T(P_0)\xrto{(p,0)}Z(X)\to 0$$
is a minimal projective presentation. Therefore $Z(X)$ is a $\tau$-rigid module if and only if $\Hom_{\Lambda}(\left(\begin{smallmatrix}
	f&0&0&0\\
	0&r&F(f)&0
\end{smallmatrix}\right),Z(X))$ is surjective. Notice that $(C,Z)$ is an adjoint pair, this is equivalent to that $\Hom_A((f,0), X)$ is surjective. Since $\Hom_A((f,0), X)=\left(\begin{smallmatrix}
\Hom_A(f,X)\\
0
\end{smallmatrix}\right)$, it follows that $Z(X)$ is a $\tau$-rigid $\Lambda$-module if and only if $X$ is a $\tau$-rigid $A$-module and $\Hom_A(Q,X)=0$ by Lemma \ref{lem:tau-rigid module}.
\end{proof}

Using Proposition \ref{prop:tau-rigid module of trivial exten}, we can get the following result.
\begin{corollary} \label{cor:T(X),Z(X) are tau-rigid}Let $\Lambda$ be the lower triangular matrix algebra
	$\left(\begin{smallmatrix}
		R&0\\
		M&S
	\end{smallmatrix}\right)$. Then
	
	 $(1)$ \cite[Proposition 4.1]{GH} $(0, M\otimes_RX)\xrto{(0,\left(\begin{smallmatrix}
			0\\
			1
		\end{smallmatrix}\right))}(X,Y\oplus M\otimes_RX)$  is a $\tau$-rigid $\Lambda$-module if and only if $X$ is a $\tau$-rigid $R$-module, $Y$ is a $\tau$-rigid $S$-module and $\Hom_S(M\otimes_RX,\tau Y)=0$.
	
	$(2)$ $(0, M\otimes_R X)\xrto{(0,0)}(X,Y)$ is a $\tau$-rigid $\Lambda$-module if and only if $X$ is a $\tau$-rigid $R$-module, $Y$ is a $\tau$-rigid $S$-module and $\Hom_S(Q,Y)=0$, where $Q$ is a projective cover of $M\otimes_RX$.
\end{corollary}

In the following we focus on the support $\tau$-tilting modules over trivial extensions of algebras.
We need the following lemma on the number of non-isomorphic indecomposable direct summands of modules.

\begin{lemma}\label{direct summands}  Let $\Lambda$ be the trivial extension of $A$ by $M$ and $X$ a left $A$-module.

$(1)$ The number of indecomposable direct summands of $X$ is equal to that of $T(X)$, that is, $|T(X)|=|X|$.

$(2)$ The number of indecomposable direct summands of $X$ is equal to that of $Z(X)$, that is, $|Z(X)|=|X|$.

$(3)$ $\Lambda\simeq T(A)$ as a $\Lambda$-module, and hence $|A|=|\Lambda|$ holds.
\end{lemma}

\begin{proof} (1) Since $CT=\mathrm{Id}$, it follows that $T$ induces an equivalence $T:\mathrm{add}X\rightarrow \mathrm{add}T(X)$ with a quasi-inverse given by $C:\mathrm{add}T(X)\rightarrow \mathrm{add}X$. In particular, this implies the number of indecomposable direct summands of $T(X)$ equals that of $X$.

(2)  Since $UZ=\mathrm{Id}$, it follows that $Z$ induces an equivalence $Z:\mathrm{add}X\rightarrow \mathrm{add}Z(X)$ with a quasi-inverse given by $U:\mathrm{add}Z(X)\rightarrow \mathrm{add}X$. In particular, this implies the number of indecomposable direct summands of $Z(X)$ equals that of $X$.

(3) This is a straight result of (1).
\end{proof}

Recall from \cite{AIR} that a pair $(X,P)\in A\mbox{-}\mathrm{mod}$ is called a {\it  $\tau$-rigid pair} if $X$ is $\tau$-rigid, $P$ is projective and $\Hom_A(P,X)=0$. Moreover, a pair $(X,P)\in A\mbox{-}\mathrm{mod}$ is called a {\it support  $\tau$-tilting pair} if it is a $\tau$-rigid pair and $|P|+|X|=|A|$. Now we are in a position to state the main results in this paper.

\begin{theorem}\label{main theorem} Let $\Lambda$ be the trivial extension of $A$ by a bimodule $M$ and let $(X,P)$ be a pair in $A\mbox{-}\mathrm{mod}$. Then
	
	$(1)$ $(T(X), T(P))$ is a support $\tau$-tilting pair in $\Lambda\mbox{-}\mathrm{mod}$ if and only if $(X,P)$ is a support $\tau$-tilting pair in $A\mbox{-}\mathrm{mod}$, $\Hom_A(P,M\otimes_A X)=0$ and $\Hom_A(M\otimes_AX, \tau X)=0$.	
	
	$(2)$ $(Z(X),T(P))$ is a support $\tau$-tilting pair in $\Lambda\mbox{-}\mathrm{mod}$ if and only if $(X,P)$ is a support $\tau$-tilting pair in $A\mbox{-}\mathrm{mod}$ and $\Hom_A(Q,X)=0$, where $Q$ is a projective cover of $M\otimes_AX$.
\end{theorem}

\begin{proof}(1) $\Leftarrow$ Since $X$ is $\tau$-rigid and $\Hom_A(M\otimes_AX, \tau X)=0$, by Proposition \ref{prop:tau-rigid module of trivial exten}, one gets $T(X)$ is $\tau$-rigid in $\Lambda\mbox{-}\mathrm{mod}$.

(a) We show $(T(X),T(P))$ is a $\tau$-rigid pair. Since $(X,P)$ is a support $\tau$-tilting pair, one gets $\Hom_A(P,X)=0$ which implies $\Hom_{\Lambda}(T(P),T(X))=0$. The assertion holds.

(b) It remains to show that $|T(P)|+|T(X)|=|\Lambda|$. By Lemma \ref{direct summands} one gets $|T(P)|=|P|$, $|T(X)|=|X|$ and $|A|=|\Lambda|$. Since $(X,P)$ is a support $\tau$-tilting pair, then $|P|+|X|=|A|$ implies that $|T(P)|+|T(X)|=|\Lambda|$.

$\Rightarrow$ Since $T(X)$ is a $\tau$-rigid module, by Proposition \ref{prop:tau-rigid module of trivial exten}, one gets that $X$ is $\tau$-rigid and $\Hom_A(M\otimes_AX, \tau X)=0$. Since $(T(X),T(P))$ is a support $\tau$-tilting pair, then $\Hom_{\Lambda}(T(P),T(X))=0$ implies $\Hom_A(P,X)=0$. By Lemma \ref{direct summands}, one gets $|T(P)|=|P|$, $|T(X)|=|X|$ and $|A|=|\Lambda|$. Then $|T(P)|+|T(X)|=|\Lambda|$ implies that $|P|+|X|=|A|$.

(2) $\Leftarrow$ Since $X$ is a $\tau$-rigid $A$-module and $\Hom_A(Q,X)=0$, where $Q$ is a projective cover of $M\otimes_AX$, one gets $Z(X)$ is $\tau$-rigid by Proposition \ref{prop:tau-rigid module of trivial exten}. We show that $(Z(X),T(P))$ is a support $\tau$-tilting pair.

(a) We show $\Hom_{\Lambda}(T(P),Z(X))=0$. Assume that we have the following commutative diagram:
$$\CD
  F(P)\oplus F^2(P) @>\alpha>> P\oplus F(P) \\
  @V (F(m),F(n)) VV @V (m,n) VV  \\
  F(X) @>0>> X
\endCD$$
where $\alpha=\left(
                \begin{smallmatrix}
                  0 & 0 \\
                  1 & 0 \\
                \end{smallmatrix}
              \right)$.
Then $(m,n)\alpha=0(F(m),F(n))$ implies $n=0$. Since $\Hom_A(P,X)=0$, we get $m=0$ and hence $(m,n)=0$. We are done.

(b) We show that $|Z(X)|+|T(P)|=|\Lambda|$. By Lemma \ref{direct summands}, one gets $|T(P)|=|P|$, $|Z(X)|=|X|$ and $|A|=|\Lambda|$. The the assertion follows from the fact $|X|+|P|=|A|$.

$\Rightarrow$ Since $Z(X)$ is $\tau$-rigid, then by Proposition \ref{prop:tau-rigid module of trivial exten} $X$ is $\tau$-rigid and $\Hom_A(Q,X)=0$, where $Q$ is the projective cover of $M\otimes_A X$. Now it suffices to show that $(X,P)$ is the support $\tau$-tilting pair.

(a) We show that $\Hom_A(P,X)=0$. By the equation $(*)$ in Section $2$, one gets that $$\Hom_\Lambda(T(P),Z(X))=\{(f,0)|\forall f\in\Hom_A(P,X)\}.$$ Since $(Z(X),T(P))$ is a support $\tau$-tilting pair, then $\Hom_\Lambda(T(P),Z(X))=0$ implies $\Hom_A(P,X)=0$.

(b) We show that $|P|+|X|=|A|$. By Lemma \ref{direct summands}, one gets $|T(P)|=|P|$, $|Z(X)|=|X|$ and $|A|=|\Lambda|$. Since $(Z(X),T(P))$ is a support $\tau$-tilting pair, then $|T(P)|+|Z(X)|=|\Lambda|$ implies that $|P|+|X|=|A|$.
\end{proof}

As a corollary, we get the following results on the support $\tau$-tilting modules over lower triangular matrix algebras which extend and generalize \cite[Theorem 4.3]{GH}.

\begin{corollary} Let $\Lambda$ be the lower triangular matrix algebra
	$\left(\begin{smallmatrix}
		R&0\\
		M&S
	\end{smallmatrix}\right)$. Then
	
	 $(1)$ $(0, M\otimes_RX)\xrto{(0,\left(\begin{smallmatrix}
			0\\
			1
		\end{smallmatrix}\right))}(X,Y\oplus M\otimes_RX)$ is a support $\tau$-tilting $\Lambda$-module if and only if $X$ is a support $\tau$-tilting $R$-module, $Y$ is a support $\tau$-tilting $S$-module, $\Hom_S(M\otimes_RX,\tau Y)=0$ and $\Hom_S(P,M\otimes_R X)=0$, where $(Y,P)$ is a support $\tau$-tilting pair.
	
	$(2)$ $(0, M\otimes X)\xrto{(0,0)}(X,Y)$ is a support $\tau$-tilting $\Lambda$-module if and only if $X$ is a support $\tau$-tilting $R$-module, $Y$ is a support $\tau$-tilting $S$-module and $\Hom_S(Q,Y)=0$, where $Q$ is a projective cover of $M\otimes_RX$.

\end{corollary}

\section{Examples}

In this section we give examples to show our main results.
\begin{example}\label{example of trivial extensions} Let $A=KQ$ with the quiver $Q:1\rightarrow 2$ and $M=\mathbb{D}A$. Let $\Lambda$ be the trivial extension of $A$ by $M$. Then

(1) $\Lambda$ is given by the quiver $Q':\xymatrix{1\ar@<.2em>[r]^{a}&2\ar@<.2em>[l]^{a}}$ with $a^3=0$, see \cite{C}.

(2) The indecomposable $\tau$-rigid modules in $A\mbox{-}\mathrm{mod}$ are $P(1),P(2),S(1)$.

(3) $T(P(1)),T(P(2))$ are indecomposable projective $\Lambda$-modules and hence $\tau$-rigid.

(4) $T(S(1))=Z(S(1))$ and $Z(S(2))$ are indecomposable non-projective $\tau$-rigid $\Lambda$-modules.

(5) The non-zero support $\tau$-tilting $A$-modules are: $P(1)\oplus P(2)$, $P(1)\oplus S(1)$, $S(1)$ and $S(2)=P(2)$.

(6) $T(P(1))\oplus T(P(2))$, $T(P(1))\oplus T(S(1))$, $T(S(1))=Z(S(1))$, $Z(S(2))$ are non-zero support $\tau$-tilting $\Lambda$-modules with the form $T(X)$ or $Z(X)$.

(7) There is a support $\tau$-tilting $\Lambda$-module $T(P(2))\oplus Z(S(2))$ which is neither of the form $T(X)$ nor of form $Z(X)$.

\end{example}

Now we give an example on the support $\tau$-tilting modules over a lower triangular matrix algebra.

\begin{example}\label{lower triangular matrix algebra} Let $A=KQ$ with the quiver $Q:1\rightarrow 2$ and $M=A$. Let $\Lambda$ be the lower triangular matrix algebra $\left(
                            \begin{smallmatrix}
                              A & 0 \\
                              A & A \\
                            \end{smallmatrix}
                          \right)$
. Then

(1) $\Lambda$ is the trivial extension of $B=\left(
                            \begin{smallmatrix}
                              A & 0 \\
                              0& A \\
                            \end{smallmatrix}
                          \right)$
by $A$.

(2) The non-zero support $\tau$-tilting $A$-modules are: $A_1=P(1)\oplus P(2)$, $A_2=P(1)\oplus S(1)$, $A_3=S(1)$ and $A_4=S(2)$.

(3) $T((A_i,0))=(0, A_i)\xrto{(0,\left(\begin{smallmatrix}
			0\\
			1
		\end{smallmatrix}\right))}(A_i,0\oplus A_i)$ is not a support $\tau$-tilting $\Lambda$-module for $i=1,2,3,4$;

$T((A_i,A_i))=(0,A_i)\xrto{(0,\left(\begin{smallmatrix}
			0\\
			1
		\end{smallmatrix}\right))}(A_i, A_i\oplus A_i)$ is a support $\tau$-tilting $\Lambda$-module for $i=1,2,3,4$;

$T((A_i,A_1))=(0,A_i)\xrto{(0,\left(\begin{smallmatrix}
			0\\
			1
		\end{smallmatrix}\right))}(A_i, A_1\oplus A_i)$ is a support $\tau$-tilting $\Lambda$-module for $i=2,3,4$;

$T((A_3,A_2))=(0,A_3)\xrto{(0,\left(\begin{smallmatrix}
			0\\
			1
		\end{smallmatrix}\right))}(A_3, A_2\oplus A_3)$ is a support $\tau$-tilting $\Lambda$-module;

$T((A_2,A_3))=(0,A_2)\xrto{(0,\left(\begin{smallmatrix}
			0\\
			1
		\end{smallmatrix}\right))}(A_2, A_3\oplus A_2)$ is not a support $\tau$-tilting $\Lambda$-module;

$T((A_i,A_4))=(0,A_i)\xrto{(0,\left(\begin{smallmatrix}
			0\\
			1
		\end{smallmatrix}\right))}(A_i, A_4\oplus A_i)$ is not a support $\tau$-tilting $\Lambda$-module for $i=1,2,3$;

$Z(A_i,0)=(0,A_i)\xrto{(0,0)}(A_i, 0)$ is a support $\tau$-tilting $\Lambda$-module for $i=1,2,3,4$;

$Z(0,A_i)=(0,0)\xrto{(0,0)}(0, A_i)=T(0,A_i)$ is a support $\tau$-tilting $\Lambda$-module for $i=1,2,3,4$;

$Z(A_4,A_3)=(0,A_4)\xrto{(0,0)}(A_4, A_3)$ is a support $\tau$-tilting $\Lambda$-module;

$Z(A_i,A_4)=(0,A_i)\xrto{(0,0)}(A_i, A_4)$ is a support $\tau$-tilting $\Lambda$-module for $i=2,3$.

\end{example}

\end{document}